\documentclass[10pt]{amsart}

\usepackage{amsfonts}
\usepackage{mathrsfs}
\usepackage{amscd}
\usepackage{amsmath}
\usepackage{amssymb}
\usepackage{latexsym}
\usepackage{lscape}
\usepackage{xypic}
\usepackage{comment}
\usepackage{amscd}
\usepackage{wasysym}  
\usepackage{stmaryrd}  
\usepackage{tikz} 
\usetikzlibrary{matrix,arrows}
\usepackage{appendix}
\usepackage{geometry}
\usepackage[nice]{nicefrac}
\usepackage{dsfont}
\usepackage{tikz-cd} 

\usepackage[pagebackref,hyperindex,linktocpage=true]{hyperref}
\hypersetup{
    colorlinks,
    linkcolor={red!50!black},
    citecolor={blue!50!black},
    urlcolor={blue!80!black}
}

\usepackage{color}

\newtheorem{thm}{Theorem}[section]

\newtheorem{prop}[thm]{Proposition}
\newtheorem{lem}[thm]{Lemma}

\newtheorem{lem-def}[thm]{Lemma-Definition}
\newtheorem{cor}[thm]{Corollary}

\theoremstyle{definition}

\newtheorem{rmk}[thm]{Remark}

\numberwithin{equation}{section}

\newcommand{\nc}{\newcommand}

\nc{\on}{\operatorname}
\nc{\fraka}{{\mathfrak a}} \nc{\bba}{{\mathbf a}}
\nc{\frakb}{{\mathfrak b}}
\nc{\frakc}{{\mathfrak c}}
\nc{\frakd}{{\mathfrak d}}
\nc{\frake}{{\mathfrak e}}
\nc{\frakf}{{\mathfrak f}}
\nc{\frakg}{{\mathfrak g}}
\nc{\frakh}{{\mathfrak h}}
\nc{\fraki}{{\mathfrak i}}
\nc{\frakj}{{\mathfrak j}}
\nc{\frakk}{{\mathfrak k}}
\nc{\frakl}{{\mathfrak l}}
\nc{\frakm}{{\mathfrak m}}
\nc{\frakn}{{\mathfrak n}}
\nc{\frako}{{\mathfrak o}}
\nc{\frakp}{{\mathfrak p}}
\nc{\frakq}{{\mathfrak q}}
\nc{\frakr}{{\mathfrak r}}
\nc{\fraks}{{\mathfrak s}}
\nc{\frakt}{{\mathfrak t}}
\nc{\fraku}{{\mathfrak u}}
\nc{\frakv}{{\mathfrak v}}
\nc{\frakw}{{\mathfrak w}}
\nc{\frakx}{{\mathfrak x}}
\nc{\fraky}{{\mathfrak y}}
\nc{\frakz}{{\mathfrak z}}
\nc{\frakA}{{\mathfrak A}}
\nc{\frakB}{{\mathfrak B}}
\nc{\frakC}{{\mathfrak C}}
\nc{\frakD}{{\mathfrak D}}
\nc{\frakE}{{\mathfrak E}}
\nc{\frakF}{{\mathfrak F}}
\nc{\frakG}{{\mathfrak G}}
\nc{\frakH}{{\mathfrak H}}
\nc{\frakI}{{\mathfrak I}}
\nc{\frakJ}{{\mathfrak J}}
\nc{\frakK}{{\mathfrak K}}
\nc{\frakL}{{\mathfrak L}}
\nc{\frakM}{{\mathfrak M}}
\nc{\frakN}{{\mathfrak N}}
\nc{\frakO}{{\mathfrak O}}
\nc{\frakP}{{\mathfrak P}}
\nc{\frakQ}{{\mathfrak Q}}
\nc{\frakR}{{\mathfrak R}}
\nc{\frakS}{{\mathfrak S}}
\nc{\frakT}{{\mathfrak T}}
\nc{\frakU}{{\mathfrak U}}
\nc{\frakV}{{\mathfrak V}}
\nc{\frakW}{{\mathfrak W}}
\nc{\frakX}{{\mathfrak X}}
\nc{\frakY}{{\mathfrak Y}}
\nc{\frakZ}{{\mathfrak Z}}
\nc{\bbA}{{\mathbb A}}
\nc{\bbB}{{\mathbb B}}
\nc{\bbC}{{\mathbb C}}
\nc{\bbD}{{\mathbb D}}
\nc{\bbE}{{\mathbb E}}
\nc{\bbF}{{\mathbb F}} \nc{\bbf}{{\mathbf f}}
\nc{\bbG}{{\mathbb G}}
\nc{\bbH}{{\mathbb H}}
\nc{\bbI}{{\mathbb I}}
\nc{\bbJ}{{\mathbb J}}
\nc{\bbK}{{\mathbb K}}
\nc{\bbL}{{\mathbb L}}
\nc{\bbM}{{\mathbb M}}
\nc{\bbN}{{\mathbb N}}
\nc{\bbO}{{\mathbb O}}
\nc{\bbP}{{\mathbb P}}
\nc{\bbQ}{{\mathbb Q}}
\nc{\bbR}{{\mathbb R}}
\nc{\bbS}{{\mathbb S}}
\nc{\bbT}{{\mathbb T}}
\nc{\bbU}{{\mathbb U}}
\nc{\bbV}{{\mathbb V}}
\nc{\bbW}{{\mathbb W}}
\nc{\bbX}{{\mathbb X}}
\nc{\bbY}{{\mathbb Y}}
\nc{\bbZ}{{\mathbb Z}}
\nc{\calA}{{\mathcal A}}
\nc{\calB}{{\mathcal B}}
\nc{\calC}{{\mathcal C}}
\nc{\calD}{{\mathcal D}}
\nc{\calE}{{\mathcal E}}
\nc{\calF}{{\mathcal F}}
\nc{\calG}{{\mathcal G}}
\nc{\calH}{{\mathcal H}}
\nc{\calI}{{\mathcal I}}
\nc{\calJ}{{\mathcal J}}
\nc{\calK}{{\mathcal K}}
\nc{\calL}{{\mathcal L}}
\nc{\calM}{{\mathcal M}}
\nc{\calN}{{\mathcal N}}
\nc{\calO}{{\mathcal O}}
\nc{\calP}{{\mathcal P}}
\nc{\calQ}{{\mathcal Q}}
\nc{\calR}{{\mathcal R}}
\nc{\calS}{{\mathcal S}}
\nc{\calT}{{\mathcal T}}
\nc{\calU}{{\mathcal U}}
\nc{\calV}{{\mathcal V}}
\nc{\calW}{{\mathcal W}}
\nc{\calX}{{\mathcal X}}
\nc{\calY}{{\mathcal Y}}
\nc{\calZ}{{\mathcal Z}}

\nc{\scrA}{{\mathscr A}}
\nc{\scrB}{{\mathscr B}}
\nc{\scrR}{{\mathscr R}}

\nc{\bnu}{{\bar{ \nu}}}

\nc{\olO}{\bar{\calO}}

\nc{\al}{{\alpha}} 
\nc{\be}{{\beta}}
\nc{\ga}{{\gamma}} \nc{\Ga}{{\Gamma}}
 \nc{\hGa}{\widehat{\Gamma}}
\nc{\ve}{{\varepsilon}} 
\nc{\la}{{\lambda}} \nc{\La}{{\Lambda}}
\nc{\om}{\omega} \nc{\Om}{\Omega} 
\nc{\sig}{{\sigma}} \nc{\Sig}{{\Sigma}}

\nc{\tnb}{\psi_{\rm tame}}
\nc{\oM}{\overline{{M}}}
\nc{\op}{{\on{op}}}
\nc{\ad}{{\on{ad}}}
\nc{\alg}{{\on{alg}}}
\nc{\Ad}{{\on{Ad}}}
\nc{\Adm}{{\on{Adm}}} \nc{\aff}{{\on{af}}}
\nc{\Aut}{{\on{Aut}}}
\nc{\Bun}{{\on{Bun}}}
\nc{\cha}{{\on{char}}}
\nc{\der}{{\on{der}}}
\nc{\Der}{{\on{Der}}}
\nc{\diag}{{\on{diag}}}
\nc{\End}{{\on{End}}}
\nc{\Fl}{{\calF\!\ell}}
\nc{\Tr}{{\on{Transp}}}
\nc{\TR}{{\calT\!\calR}}
\nc{\Gal}{{\on{Gal}}}
\nc{\Gr}{{\on{Gr}}}
\nc{\rH}{{\on{H}}}
\nc{\Hom}{{\on{Hom}}}
\nc{\IC}{{\on{IC}}}
\nc{\id}{{\on{id}}}
\nc{\Id}{{\on{Id}}}
\nc{\ind}{{\on{ind}}}
\nc{\Ind}{{\on{Ind}}}
\nc{\Lie}{{\on{Lie}}}
\nc{\Pic}{{\on{Pic}}}
\nc{\pr}{{\on{pr}}}
\nc{\Res}{{\on{Res}}}
\nc{\res}{{\on{res}}} \nc{\Sat}{{\on{Sat}}}
\nc{\s}{{\on{sc}}}
\nc{\drv}{{\on{der}}}
\nc{\sgn}{{\on{sgn}}}
\nc{\Spec}{{\on{Spec}}}\nc{\Spf}{\on{Spf}} 
\nc{\Sph}{\on{Sph}}
\nc{\St}{{\on{St}}}
\nc{\tr}{{\on{tr}}}
\nc{\Mod}{{\mathrm{-Mod}}}
\nc{\Hilb}{{\on{Hilb}}} 
\nc{\Ext}{{\on{Ext}}} 
\nc{\vs}{{\on{Vec}}}
\nc{\ev}{{\on{ev}}}
\nc{\nO}{{\breve{\calO}}}
\nc{\tS}{{\tilde{S}}}
\nc{\spe}{{\on{sp}}}
\nc{\loc}{{\on{loc}}}

\nc{\nscrR}{{\mathscr{R}^{\on{nr}}}}

\nc{\GL}{{\on{GL}}}
\nc{\U}{{\on{U}}}
\nc{\Gl}{\on{Gl}} 
\nc{\GSp}{{\on{GSp}}}
\nc{\gl}{{\frakg\frakl}}
\nc{\SL}{{\on{SL}}} 
\nc{\SU}{{\on{SU}}} 
\nc{\SO}{{\on{SO}}}
\nc{\PGL}{{\on{PGL}}}

\nc{\Conv}{{\on{Conv}}}
\nc{\Rep}{{\on{Rep}}}
\nc{\Dom}{{\on{Dom}}}
\nc{\red}{{\on{red}}}
\nc{\act}{{\on{act}}}
\nc{\nr}{{\on{nr}}}
\nc{\ctf}{{\on{ctf}}}

\nc{\str}{{\on{-}}} 
\nc{\os}{{\bar{s}}}
\nc{\oeta}{{\bar{\eta}}}

\nc{\hookto}{\hookrightarrow}
\nc{\longto}{\longrightarrow}
\nc{\leftto}{\leftarrow}
\nc{\onto}{\twoheadrightarrow}
\nc{\lonto}{\twoheadleftarrow}

\nc{\uG}{{\underline{G}}}
\nc{\uA}{{\underline{A}}}
\nc{\uS}{{\underline{S}}}
\nc{\uT}{{\underline{T}}}
\nc{\uM}{{\underline{M}}}
\nc{\uP}{{\underline{P}}}
\nc{\uB}{{\underline{B}}}
\nc{\uN}{{\underline{N}}}

\nc{\ucG}{{\underline{\calG}}}
\nc{\ucA}{{\underline{\calA}}}
\nc{\ucS}{{\underline{\calS}}}
\nc{\ucT}{{\underline{\calT}}}
\nc{\ucM}{{\underline{\calM}}}
\nc{\ucP}{{\underline{\calP}}}
\nc{\ucN}{{\underline{\calN}}}

\nc{\bF}{{\breve{F}}}

\nc{\oFl}{{\overline{\Fl}}} 
\nc{\bU}{{\overline{U}}}
\nc{\tGr}{{\tilde{\Gr}}}
\nc{\cGr}{\calG\! r}
\nc{\oGr}{\overline{\on{Gr}}} 
\nc{\ocGr}{\overline{\calG\! r}}
\nc{\co}{{\colon}}
\nc{\sch}[1]{(Sch/{#1})}
\nc{\HypLoc}[1]{HypLoc({#1})}

\nc{\ohtimes}{\stackrel{!}{\otimes}}
\nc{\boxtilde}{\widetilde{\boxtimes}}
\nc{\vstar}{{\varhexstar}}

\nc{\Div}{\on{Div}}

\nc{\bslash}{\backslash}
\nc{\algQl}{{\bar{\bbQ}_\ell}}
\nc{\sF}{{\bar{F}}}
\nc{\nF}{{\breve{F}}}
\nc{\nW}{{W^{\on{nr}}}}
\nc{\sk}{{\bar{k}}}
\nc{\cont}{\on{c}}
\nc{\Supp}{\on{Supp}}
\nc{\blt}{\bullet}  
\nc{\dom}{\on{dom}}
\nc{\scon}{{\on{sc}}} 
\nc{\Affine}{\on{Aff}} 
\nc{\nscrA}{\mathscr{A}^{\on{nr}}} 
\nc{\nfraka}{{\bbf^{\on{nr}}}}
\nc{\ran}{{\rangle}}
\nc{\lan}{{\langle}}
\nc{\bk}{{\bar{k}}}
\nc{\tF}{{\tilde{F}}}
\nc{\sS}{{\bar{S}}}
\nc{\LG}{{^\text{L}\hspace{-0.04cm}G}}
\nc{\LL}{{^\text{L}\hspace{-0.07cm}L}}

\nc{\pot}[1]{ [\hspace{-0,5mm}[ {#1} ]\hspace{-0,5mm}] }
\nc{\rpot}[1]{ (\hspace{-0,7mm}( {#1} )\hspace{-0,7mm}) }

\nc{\defined}{\hspace{0.1cm}\stackrel{\text{\tiny \rm def}}{=}\hspace{0.1cm}}

\begin{document}

\title[Semisimple local Langlands parameters and parahoric Satake parameters]{Compatibility of semisimple local Langlands parameters with parahoric Satake parameters}
\author{Qihang Li}

\address{Department of Mathematics, University of Maryland, College Park, MD 20742-4015, DC, USA}
\email{mathlqh@umd.edu}

\thanks{Research of Q.L. is partially supported by a Hauptman Summer Fellowship (2021) at University of Maryland.}

\maketitle

\begin{abstract}
In this paper, we prove that there is at most one correspondence between parahoric-spherical representations and semisimple local Langlands parameters which satisfies certain natural properties. Our proof of this uniqueness statement is very formal. In particular, the semisimple local Langlands parameters constructed in \cite{FS21} yield the unique candidate for such representations. As a corollary, we prove the conjecture \cite[Conjecture 13.1]{Hai15} which posits the compatibility of semisimple local Langlands parameters with parahoric Satake parameters.
\end{abstract}

\setcounter{tocdepth}{1}
\tableofcontents
\setcounter{section}{0}

\thispagestyle{empty}

\section{Introduction}

\subsection{Formulation of the main result} Let $F$ be a nonarchimedean local field with residue field $k_F$ of cardinality $q=p^m$. The ring of integers of $F$ is denoted by $\mathcal{O}_F$. Let $\sF/F$ be a separable closure. We denote by $\Ga_F$ the Galois group with inertia subgroup $I_F$ and by $W_F$ the Weil group of $F$. We fix a geometric Frobenius lift $\Phi_F\in\ W_F$. For $G$ a connected reductive group defined over $F$, the Langlands dual group $\widehat{G}$ of $G$ considered in this paper is always defined over the complex number field $\mathbb{C}$. Since the Langlands parameters constructed in \cite{FS21} are considered to be with $\ell$-adic coefficients, we fix an isomorphism $\mathbb{C} \cong \overline{\mathbb{Q}}_\ell$ throughout this paper to make everything work over $\mathbb{C}$. By ${^L\! G \ := \widehat{G} \rtimes W_F}$ we mean the algebraic $L$-group of $G$.

Let $\Pi(G/F)$ be the set consisting of smooth irreducible representations of $G(F)$ up to isomorphism and let $\Phi(G/F)$ be the set consisting of admissible homomorphisms $W_F'\rightarrow$ ${^L\! G}$ up to $\widehat{G}$-conjugacy. Here $W_F' \overset{\underset{\mathrm{def}}{}}{=} W_F \ltimes \mathbb{C}$ is the Weil-Deligne group of $F$.

A version of local Langlands correspondence is to construct a finite-to-one map $LLC_G:\Pi(G/F) \rightarrow \Phi(G/F)$ satisfying some prescribed properties.

In \cite{FS21}, Fargues and Scholze construct a solution to the semisimplified version of the above question. More specifically, let $\Theta(G/F)$ be the set of smooth (or, continuous) homomorphisms $W_F\rightarrow $ ${^L\! G}$ with semisimple image up to $\widehat{G}$-conjugacy (elements in this set are called sesmisimple local Langlands parameters in this paper). A correspondence $LLC_G^{FS}:\Pi(G/F)\rightarrow\Theta(G/F)$ (any map of this form will be called semisimple local Langlands correspondence in this paper) is constructed in \cite{FS21} and it satisfies properties listed in \cite[Theorem I.9.6]{FS21}.

A natural question to ask is if $LLC_G^{FS}:\pi\mapsto\varphi_{\pi}^{FS}$ agrees with previous constructions given by other people which also provide full or partial candidates for the local Langlands correspondence. For $G=GL_n$, this has been done in \cite[Theorem I.9.6]{FS21}. Furthermore, when $F$ is a $p$-adic field, the case of inner forms of $GL_n$ was settled in \cite{HKW22}, and the case of inner forms of $GSp_4$ (with some assumptions on $F$) was done in \cite{Ham21}. More recently, results for some unitary groups also have been obtained in \cite{HMN22}.

Solutions to the question about comparing local Langlands correspondence mentioned above all started from some given group $G$ and then compared the images of all possible smooth irreducible representations of $G(F)$ under a priori different correspondences. In this paper, we approach the question from a slightly different direction. We deal with all connected reductive groups $G$ simultaneously, but we focus only on a proper subset of $\Pi(G/F)$.
To be more precise, we are interested in smooth irreducible representations $\pi$ of $G(F)$ with non-trivial parahoric fixed vectors (representations of this kind are called parahoric-spherical representations in this paper). In other words, any representation $\pi$ in this subset satisfies the condition that $\pi^J \neq 0$ for some parahoric subgroup $J \subseteq G(F)$. One special property (or equivalent definition) of such representations is that they are irreducible subquotients of parabolic inductions of weakly unramified characters of some minimal Levi subgroup. A parahoric Satake parameter $s(\pi)$ is attached to every representation $\pi$ of this kind in \cite{Hai15}, so a natural question is to compare $s(\pi)$ and $\varphi_\pi^{FS}(\Phi_F)$, which is the conjecture \cite[Conjecture 13.1]{Hai15}. This will be proved in our Corollary B. Firstly we will prove the following theorem, which says that for smooth irreducible representations with non-trivial parahoric fixed vectors, there is at most one semisimple local Langlands correspondence (the existence is guaranteed by $\varphi^{FS}_{\pi}$) if the correspondence satisfies some formal properties contained in the list of \cite[Theorem I.9.6]{FS21} (we need a stronger version for the last property in the list). This is our Main Theorem. Note that since we have to deal with Weil restriction of scalars, we also need to consider finite separable extensions of $F$.

\bigskip

\noindent\textbf{Main Theorem.}
{\it  Consider a family of maps $\pi \mapsto \varphi_{\pi}^{ss}$ from parahoric-spherical representations of $G(E)$ \textup{(}for varying finite separable extensions $E/F$ and varying connected reductive groups $G$ over $E$\textup{)} to semisimple local Langlands parameters, satisfying the following properties:

\noindent \textup{(i)} For $G=T$ is a torus, $\pi \mapsto \varphi_{\pi}^{ss}$  is the usual local Langlands correspondence for tori.

\noindent \textup{(ii)} The correspondence $\pi \mapsto \varphi_{\pi}^{ss}$  is compatible with twists by weakly unramified characters. We note that for this property, we are only considering characters of $(G/G_{der})(E)$ and this is the setting in which the property is proved at the end of \cite[IX.6.4]{FS21}. Let $\pi$ be an irreducible smooth representation of $G(E)$ and $\chi$ be a weakly unramified character of $(G/G_{der})(E)$ which we also view as a character of $G(E)$. We denote by $\pi\otimes\chi$ the irreducible smooth representation of $G(E)$ obtained by twisting $\pi$ by $\chi$. This compatibility means that the semisimple Langlands parameter $\varphi_{\pi\otimes\chi}^{ss}$ is the same as $(\varphi_{\pi}^{ss},\varphi_{\chi}^{ss})': W_E\rightarrow \! ^L\! (G\times (G/G_{der}))$ via the map $\! ^L\! (G\times (G/G_{der})) \rightarrow \! ^L\! G$ induced from the diagonal map $G\rightarrow G\times (G/G_{der})$.

\noindent \textup{(iii)} If $G'\rightarrow G$ is a map of reductive groups inducing an isomorphism of adjoint groups, $\pi$ is an irreducible smooth representations of $G(E)$ and $\pi'$ is an irreducible subquotient of $\pi|_{G'(E)}$, then $\varphi_{\pi'}^{ss}$ is given by the composition of the induced map between the $L$-groups $^L\! G\rightarrow $ $^L\! G'$ and $\varphi_{\pi}^{ss}$.

\noindent \textup{(iv)} The correspondence $\pi \mapsto \varphi_{\pi}^{ss}$ is compatible with direct product. Let $G=G1\times G2$ be a product of two groups over $E$. Let $\pi_i$ be an irreducible smooth representation of $G_i(E)$ and denote by $\pi_1 \otimes \pi_2$ the corresponding irreducible smooth representation of $G(E)$. This compatibility means that the semisimple Langlands paramter $\varphi_{\pi_1 \otimes \pi_2}^{ss}$ is the same as $(\varphi_{\pi_1}^{ss}.\varphi_{\pi_2}^{ss}): W_E\rightarrow \!^L\! G_1\times \!^L\! G_2$ via the natural map $\!^L\! G\rightarrow \!^L\! G_1\times \!^L\! G_2$.

\noindent \textup{(v)} The correspondence $\pi \mapsto \varphi_{\pi}^{ss}$ is compatible with Weil restriction of scalars. Let $G= Res_{E'/E}G'$ be the Weil restriction of scalars of a reductive group $G'$ over some finite separable extension $E'/E$. Let $\pi$ be an irreducible smooth representation of $G(E)$ and denote by $\pi'$ the corresponding irreducible smooth representation of $G'(E')$. This compatibility means that $\varphi_{\pi}^{ss}$ and $\varphi_{\pi'}^{ss}$ are the same via the canonical identification $Z^1(W_E,\widehat{G})\cong Z^1(W_{E'},\widehat{G'})$.

\noindent \textup{(vi)} The correspondence $\pi \mapsto \varphi_{\pi}^{ss}$ is compatible with parabolic induction. Let $P$ be a parabolic subgroup of $G$ with Levi component $M$. Let $\sigma$ be an irreducible smooth representation of $M(E)$. This compatibility means that for any irreducible subquotient $\pi$ of the normalized parabolic induction $\textup{Ind}_P^G\sigma$, its semisimple local Langlands parameter $\varphi_\pi^{ss}$ is given by the composition $\varphi_{\sigma}^{ss}$ with the $L$-embedding $^L\! M \hookrightarrow ^L\! G $. It is worth mentioning that the proof in \cite[IX.7.2]{FS21} deals with unnormalized parabolic induction and on the Galois side the cyclotomic twist is involved when they consider the embedding $^L\! M \hookrightarrow ^L\! G $, which implies the results for the normalized parabolic inductions.

Then the family of maps $\pi \mapsto \varphi_{\pi}^{ss}$ is uniquely determined by its values on the trivial character of $PGL_{1,D_E}$ for all finite separable extensions $E/F$ and finite dimensional $E$-central division algebras  $D_E$. In particular, there is at most one family of correspondences $\pi \mapsto \varphi_{\pi}^{ss}$ which satisfies the above properties along with property \textup{(vii)} below.

\noindent \textup{(vii)} Let $G$ be any inner form of $GL_n$, the correspondence $\pi \mapsto \varphi_{\pi}^{ss}$ agrees with the usual \textup{(}semisimplified\textup{)} local Langlands correspondence for $G$.

Furthermore, the same result holds if we replace $\varphi_{\pi}^{ss}$ with its value at $\Phi_E$ and we require that $\varphi_{\pi}^{ss}(\Phi_E)\in [\widehat{G}^{I_E}\rtimes \Phi_E]_{\textup{ss}}/\widehat{G}^{I_E}$ where \textup{ }$\textup{ss}$\textup{ } means the subset consisting of semisimple elements.
}

\bigskip

Some of the assumptions in the Main Theorem can be weakened (see Remark \ref{weaken}). When $F$ is a $p$-adic field (a finite extension of $\mathbb{Q}_p$), the semisimple local Langlands correspondence $LLC_G^{FS}$constructed in \cite{FS21} satisfies all the properties in the Main Theorem, and actually that is why we are considering these properties here. Properties (i)-(vi) have been verified in \cite[Theorem I.9.6]{FS21}. In the case $G=GL_n$, property (vii) is also verfied in the same theorem. For general inner forms of $GL_n$, it is proved in \cite[Theorem 6.6.1]{HKW22} (this is the only place where we need to assume that $F$ is a $p$-adic field). Now applying the Main Theorem to $\pi \mapsto \varphi_{\pi}^{FS}$, we get the following corollary.

\bigskip

\noindent\textbf{Corollary A.} {\it Assume that $F$ is a $p$-adic field. For smooth irreducible representations $\pi$ with non-trivial parahoric fixed vectors, the semisimple local Langlands correspondence $\pi\mapsto \varphi_{\pi}^{FS}$ constructed in \cite{FS21} is the unique one satisfying all the properties \textup{(}property \textup{(vii)} included\textup{)} mentioned in the Main Theorem.
}
\bigskip

Once we are able to check that $\pi\mapsto s(\pi)$ also satisfies all the properties mentioned in the Main Theorem (we will do this in \textsection 5 and there is one more extra step), we can obtain the following corollary in a similar way, proving the conjecture \cite[Conjecture 13.1]{Hai15}.

\bigskip

\noindent\textbf{Corollary B.} {\it Assume that $F$ is a $p$-adic field. For any smooth irreducible representation $\pi$ with non-trivial parahoric fixed vectors, the value of $\varphi_{\pi}^{FS}$ at $\Phi_F$ agrees with the parahoric Satake parameter $s(\pi)$ constructed in \cite{Hai15}.
}

\subsection{Overview} In \textsection{2}, we recall the construction of $s(\pi)$ and set up notions for later sections, then we prove Corollary B in the case of quasi-split groups. In \textsection 3, by using the same strategy as in \textsection 2, we identify Genestier-Lafforgue parameters with Fargues-Scholze parameters for principal series representations of quasi-split groups. In \textsection{4}, we give a proof of the Main Theorem by using several reductions. In \textsection{5}, we deal with Corollary B for general $G$ and we also prove some results about $\varphi^{FS}_{\pi}(I_F)$.

\subsection{Acknowledgements} It is my pleasure to thank my advisor Thomas J. Haines for giving me this project, for several useful discussions about it and also for helpful comments on the preliminary version. I also really appreciate the kind suggestions given by very responsible reviewer and editor who make this paper much more understandable. My research is partially supported by a Hauptman Summer Fellowship (2021) at University of Maryland and I thank the donor Carol Fullerton for her kind support sincerely.

\subsection{Conventions on Parahoric subgroups and Notations} We use the same conventions on parahoric subgroups as in \cite{Hai15} (which are also the same as the conventions in \cite{BT84}). Parahoric subgroups are understood to be the $\mathcal{O}_F$-points of a \textbf{connected} group scheme over $\mathcal{O}_F$. For a connected reductive group $G$ defined over $F$, we denote by $G_{der}$ its algebraic derived group, by $G_{ab}$ its maximal quotient torus $G/G_{der}$, by $Z(G)$ its center, and by $G_{ad}$ its adjoint group. We denote by $D$ some finite dimensional central division algebra over some field which should be clear from the context. The algebraic group $GL_{1,D}$ is understood to be the inner from of $GL_n$ ($n$ is the index of $D$) defined by units of $D$ and $SL_{1,D} \subseteq GL_{1,D}$ is defined to be the subgroup of $GL_{1,D}$ of reduced norm $1$. We also denote by $PGL_{1,D}$ the adjoint group of $GL_{1,D}$ and $SL_{1,D}$. By $F^{un}$ we mean the maximal unramified extension of $F$, and $\breve{F}$ is the completion of $F^{un}$. All parabolic inductions used in this paper are normalized. 

 \section{The compatibility with parahoric Satake parameters: Quasi-split case}\label{Compatibility_quasisplit}
 Let $\pi$ be a smooth irreducible representation of $G(F)$ with non-trivial parahoric fixed vectors. We briefly recall how the parahoric Satake parameter $s(\pi)$ is constructed in \cite{Hai15}. For general $G$, according to \cite[\textsection 11.5]{Hai14}, the supercuspidal support of $\pi$ is a pair $(M,\chi)_G$ with $M$ being a minimal $F$-Levi subgroup of $G$ (equivalently $M$ is the centralizer of some maximal $F$-split torus $A\subseteq G$) and $\chi$ is a weakly unramified character of $M(F)$. Here weakly unramified characters are understood to be those characters vanishing on $M(F)_1$ which is the intersection of the kernel of the Kottwitz homomorphism $\kappa_M$ \cite[\textsection 7]{Ko97} and $M(F)$. One gets an isomorphism via the Kottwitz homomorphism
 $$
 \kappa_M:M(F)/M(F)_1\cong X^*(Z(\widehat{M})^{I_F}_{\Phi_F})
 $$
 Since $\chi\in \Hom_{grp}(M(F)/M(F)_1,\mathbb{C}^{\times})$, it can be viewed naturally as an element in $(Z(\widehat{M})^{I_F})_{\Phi_F}$. When $G$ is quasi-split (equivalently when $M=T$ is a maximal torus of $G$), the parahoric Satake parameter $s(\pi)$ is defined to be the image of the corresponding element under the map 
 $$
 (\widehat{T}^{I_F})_{\Phi_F}\rightarrow (\widehat{T}^{I_F})_{\Phi_F}/W(G,A)\rightarrow {{^L\! G}}/\widehat{G}
 $$
 Here $W(G,A)$ is the relative Weyl group and one can show that $s(\pi)$ is independent of the choice of $(M,\chi)$. When $G$ is not quasi-split, the construction of $s(\pi)$ is more complicated and we will review it in \textsection{5}.
 
 We firstly deal with the case when $G=T$ is a torus. For that, we need a lemma which is implicitly used in \cite{Hai15}.
 \begin{lem} \label{tori_lemma}
Denote by $\pi \mapsto \varphi_{\pi}$ the usual local Langlands correspondence for tori. If $G=T$ is a torus, then $s(\pi)=\varphi_{\pi}(\Phi_F)$.
 \end{lem}
 
 \begin{proof}
This follows from the compatibility of Langlands duality for quasicharacters with the Kottwitz homomorphism. In the case of tori, Langlands duality for quasicharacters is the following isomorphism
$$
\Hom_{cont}(T(F),\mathbb{C}^\times)\cong H^1(W_F,\widehat{T})
$$
Under this isomorphism, the set of weakly unramified characters $\Hom_{grp}(T(F)/T(F)_1,\mathbb{C}^{\times})\subseteq \Hom_{cont}(T(F),\mathbb{C}^\times)$ is identified with $H^1(\langle \Phi_F \rangle,\widehat{T}^{I_F})\subseteq  H^1(W_F,\widehat{T})$ (viewed as a subset via the inflation map), so we obtain
$$
\Hom_{grp}(T(F)/T(F)_1,\mathbb{C}^{\times}) \cong H^1(\langle \Phi_F \rangle,\widehat{T}^{I_F})
$$
Now the right hand side is naturally identified with $(\widehat{T}^{I_F})_{\Phi_F}$ by considering the image of $\Phi_F$ under any 1-cocycle representative. Finally we get
$$
\Hom_{grp}(T(F)/T(F)_1,\mathbb{C}^{\times}) \cong (\widehat{T}^{I_F})_{\Phi_F}.
$$
We notice that the Kottwitz homomorphism has the same source and target, and actually it turns out that they agree in this case. For more discussions about this, see \cite[\textsection3.3.1]{Hai14} and also the reference used there. We also note that our normalization of local class field theory sends a uniformizer in $F^\times$ to the image of $\Phi_F$ in $W^{ab}_F$ (the abelianization of $W_F$), so both the Kottwitz homomorphism and Langlands duality for quasicharacters should be normalized in this way as well.
 \end{proof}
 
 \begin{thm} \label{tori_thm}
 If $G=T$ is a torus, then $s(\pi)=\varphi_{\pi}^{FS}(\Phi_F)$.
 \end{thm}
 
 \begin{proof}
For tori, $\pi \mapsto \varphi_{\pi}^{FS}$ is the usual local Langlands correspondence. Now by Lemma \ref{tori_lemma}, we know that $s(\pi)=\varphi_{\pi}(\Phi_F)=\varphi_{\pi}^{FS}(\Phi_F)$.
 \end{proof}
 
\begin{cor}
For any quasi-split $G$, $s(\pi)=\varphi_{\pi}^{FS}(\Phi_F)$.
\end{cor}

\begin{proof}
When $\pi$ has non-trivial parahoric fixed vectors, its cuspidal support is a pair $(M,\chi)$ where $M$ is a minimal $F$-Levi subgroup of $G$. Since both $\pi \mapsto \varphi_{\pi}^{FS}$ and $\pi \mapsto s(\pi)$ are compatible with parabolic induction (for $s(\pi)$, this is the definition), we only need to consider the result for $M$ which is a torus in this case, then the result follows from Theorem \ref{tori_thm}.
\end{proof}

\section{A Simple Version of the Main theorem: Quasi-split case }

When $F$ is of equal characteristic, Genestier and Lafforgue constructed a semisimple local Langlands correspondence in \cite{GL18} which we denote by $\pi \mapsto \varphi_{\pi}^{GL}$. The construction $\pi \mapsto \varphi_{\pi}^{FS}$ also works in this case, so a natural question is to compare $\varphi_{\pi}^{GL}$ with $\varphi_{\pi}^{FS}$. Following the same strategy in \textsection 2, we get the following theorem.

\begin{thm} \label{thm_quasisplit}
Consider a family of maps $\pi \mapsto \varphi_{\pi}^{ss}$ from parahoric-spherical representations of $G(F)$ \textup{(}for varying \textbf{quasi-split} connected reductive groups $G$ over $F$\textup{)} to semisimple local Langlands parameters. There is at most one family of correspondences which satisfies properties below.

\noindent \textup{(i)} For $G=T$ is a torus, $\pi \mapsto \varphi_{\pi}^{ss}$  is the usual local Langlands correspondence.

\noindent \textup{(ii)} The correspondence $\pi \mapsto \varphi_{\pi}^{ss}$ is compatible with parabolic induction.

\end{thm}

\begin{proof}
This is exactly the same as the proof in \textsection 2.   
\end{proof}

\begin{rmk}
Note that in the statement of Theorem \ref{thm_quasisplit}, we do not need to consider finite separable extensions $E/F$.
\end{rmk}

\begin{cor} \label{FSGL}
Let $G$ be a \textbf{quasi-split} connected reductive group defined over $F$ and let $\pi$ be a smooth irreducible representation of $G(F)$ with non-trivial parahoric fixed vectors. Then $\varphi^{FS}_{\pi}=\varphi^{GL}_{\pi}$.
\end{cor}

\begin{proof}
Both $\varphi^{FS}_{\pi}$ and $\varphi^{GL}_{\pi}$ satisfy the conditions in Theorem \ref{thm_quasisplit}. For $\varphi^{FS}_{\pi}$, this is \cite[Theorem I.9.6]{FS21}. For $\varphi^{GL}_{\pi}$, this is \cite[Théorème 0.1]{GL18}.
\end{proof}

Actually the results in Theorem \ref{thm_quasisplit} and Corollary \ref{FSGL} still hold when we consider principal series representations (in the sense that it is realized as an irreducible subquotient of the parabolic inductions of a character of a Cartan subgroup) instead. We record it as the following proposition.

\begin{prop}
Consider a family of maps $\pi \mapsto \varphi_{\pi}^{ss}$ from principal series representations of $G(F)$ \textup{(}for varying \textbf{quasi-split} connected reductive groups $G$ over $F$\textup{)} to semisimple local Langlands parameters. There is at most one family of correspondences which satisfies the properties below.

\noindent \textup{(i)} For $G=T$ is a torus, $\pi \mapsto \varphi_{\pi}^{ss}$  is the usual local Langlands correspondence.

\noindent \textup{(ii)} The correspondence $\pi \mapsto \varphi_{\pi}^{ss}$ is compatible with parabolic induction.

In particular, for any principal series representation $\pi$, we have $\varphi^{FS}_{\pi}=\varphi^{GL}_{\pi}$.
\end{prop}

\begin{proof}
This is exactly the same as the proof in \textsection 2 and the proof in Corollary \ref{FSGL}.
\end{proof}

\section{Proof of the Main Theorem}\label{Proof_of_main_theorem}

In this section, we give a proof of the Main Theorem for $\pi \mapsto \varphi^{ss}_{\pi}$ (for $\pi \mapsto \varphi^{ss}_{\pi}(\Phi_F)$, the proof is similar). Our aim is to show that the semisimple local Langlands parameter $\varphi^{ss}_\pi$ is uniquely determined by the properties (i)-(vii). The proof consists of several reduction steps. In every subsection, we explain how one can reduce to the corresponding case. To begin with, we recall that we are supposed to consider any parahoric-spherical representation $\pi$ for any connected reductive group over $F$. 

\subsection{Any weakly unramified character of any anisotropic-mod-center group} \ 

\bigskip
This reduction follows from the argument we used several times. By the compatibility with parabolic induction (property (vi)), it suffices to consider the cuspidal support $(M,\chi)_G$ of $\pi$. Here $\chi$ is a weakly unramified charater of $M(F)$ \cite[\textsection 11.5]{Hai14}. Since $M$ is a minimal $F$-Levi subgroup of $G$, it is $F$-anisotropic mod center. Thus, it is enough to treat only the case when $G$ is an anisotropic-mod-center group and $\pi$ is a weakly unramified character.

\subsection{Any weakly unramified character of any anisotropic-mod-center group with simply-connected derived group} \ 

\bigskip
For this step, let $1\rightarrow Z \rightarrow G'\rightarrow G \rightarrow 1$ be a $z$-extension ($Z$ is an induced torus and $G'$ has simply-connected derived group). If $G$ is $F$-anisotropic mod center, then $G'$ is also $F$-anisotropic mod center. Since the Kottwitz homomorphism is functorial, $G'(F)_1$ is mapped into $G(F)_1$, thus we know that weakly unramified characters of $G(F)$ pull back to weakly unramified characters of $G'(F)$.
Now since $G'$ and $G$ share the same adjoint group, one can use the compatibility for this situation (property (iii) in the Main Theorem). To reduce the problem to $G'$, it remains to set up an injection between $L$-parameters (which are classified by $H^1(W_F,-)$) for different groups and this is the following lemma.

\begin{lem} \label{inj_lem}   Let $1\rightarrow Z \rightarrow G'\rightarrow G \rightarrow 1$ be a $z$-extension, then the natural map $H^1(W_F,\widehat{G}) \rightarrow H^1(W_F,\widehat{G'})$ is an injection.
\end{lem}

\begin{proof}
By Shapiro's lemma, we know that $H^1(F,Z)=0$ and this implies that the map $G'(F) \rightarrow G(F)$ is surjective, and thus the induced map $\Hom_{cont}(G(F),\mathbb{C}^{\times}) \rightarrow \Hom_{cont}(G'(F),\mathbb{C}^{\times})$ is an injection. By Langlands duality for quasicharacters (which has been used in Lemma \ref{tori_lemma}), we know that the natural map $H^1(W_F,Z(\widehat{G})) \rightarrow H^1(W_F,Z(\widehat{G'}))$ is an injection. Now from the exact sequence 
$$1\rightarrow Z \rightarrow G'\rightarrow G \rightarrow 1,$$
we get an exact sequence of $W_F$-modules (see \cite[1.8]{Ko84} for more details):
$$1 \rightarrow  Z(\widehat{G}) \rightarrow Z(\widehat{G'}) \rightarrow \widehat{Z} \rightarrow 1.$$
Since we know that $H^1(W_F,Z(\widehat{G})) \rightarrow H^1(W_F,Z(\widehat{G'}))$ is an injection, we obtain the fact that $Z(\widehat{G'})^{W_F} \rightarrow \widehat{Z}^{W_F}$ is a surjection. In other words, any $W_F$-invariant element in $\widehat{Z}$ can be lifted to a $W_F$-invariant element in the center of $\widehat{G'}$. However, from the exact sequence \cite[10.2]{Bo79}
$$1 \rightarrow  \widehat{G} \rightarrow \widehat{G'} \rightarrow \widehat{Z} \rightarrow 1$$
and general results about group cohomology, this implies that the natural map $H^1(W_F,\widehat{G}) \rightarrow H^1(W_F,\widehat{G'})$ is an injection (fibers of this map are orbits of the action of $\widehat{Z}^{W_F}$ on $H^1(W_F,\widehat{G})$ \cite[Proposition 39]{Ser97} and they are all singletons in this case).

\end{proof}

Now we are done since $G'$ is of the desired form.

\subsection{The trivial character of any anisotropic-mod-center group with simply-connected derived group} \ 

\bigskip
We want to apply the compatibility with character twist in this step. For that, we need the following lemma.
\begin{lem} \label{sc_lemma}
 Let $G$ be a connected reductive group defined over $F$. Suppose the algebraic derived group $G_{der}$ is simply connected. Then any weakly unramified character of $G(F)$ factors through $G_{ab}(F)$ \textup{(}$G_{ab}=G/G_{der}$\textup{)}.
\end{lem}

\begin{proof}
 Since $G_{der}$ is semisimple and simply-connected, we know that $H^1(F,G_{der})=0$ from \cite[REMARQUES 3.16]{BT87}, and thus the map $G(F) \rightarrow G_{ab}(F)$ is a surjection. Now it suffices to show that $G_{der}(F)\subset G(F)_1$. But in the definition of the Kottwitz homomorphism, when $G_{der}$ is simply connected, the map is given by $G(\breve{F})\rightarrow G_{ab}(\breve{F})\rightarrow X^*(\widehat{G_{ab}})_{I_F}=X^*(Z(\widehat{G}))_{I_F}$, and thus $G_{der}(F)$ must be contained in the kernel of the Kottwitz homomorphism.
\end{proof}
 
\begin{rmk}
 Note that the assumption on the simply-connectedness is necessary in Lemma \ref{sc_lemma}. When $G_{der}$ is not simply-connected, even if $G$ is $F$-anisotropic mod center, the lemma can still be false. Here is a counterexample. Let $D$ be the quaternion algebra over $F$, then $GL_{1,D}$ is an inner form $GL_{2,F}$ and $G:=PGL_{1,D}$ (the adjoint group of $GL_{1,D}$) is an inner form of $PGL_{2,F}$. One can check that $(PGL_{1,D})_{der} =PGL_{1,D}$ (check over algebraically closed fields). Suppose the lemma were true in this case, then we would have $G(F)=G_{der}(F)=G(F)_1$. Thus, via the Kottwitz homomorphism, we know that $G(F)/G(F)_1$ is isomorphic to $X^*(Z(\widehat{G}))^{\Phi_F}_{I_F}$, which implies that $X^*(Z(\widehat{G}))^{\Phi_F}_{I_F}$ is trivial. Since $G$ is an inner form of the  split group $PGL_{2,F}$, the Galois group acts trivially on $\widehat{G}$ and thus also trivially on $Z(\widehat{G})$, but then $X^*(Z(\widehat{G}))^{\Phi_F}_{I_F}=X^*(Z(\widehat{G}))=\mathbb{Z}/2\mathbb{Z}$ which is not trivial. This is a contradiction.
\end{rmk}

\begin{rmk}
  If one is only interested in unramified characters, then the assumption on the derived group is not necessary. One definition of unramified characters $\chi$ is that $\Lambda_0\subset Ker(\chi)$ where $\Lambda_0$ is the intersection of the kernels of all $F$-rational characters of $G$. Since then $G_{der}(F)$ is automatically contained in $\Lambda_0$, the result follows.
\end{rmk}
 
Now the reduction follows from the compatibility with character twist (property (ii)) and the compatibility with the usual local Langlands correspondence in the case of tori (property (i)).
 




\subsection{The trivial character of any anisotropic group of adjoint type} \ 

\bigskip
For the trivial character of any anisotropic-mod-center group with simply-connected derived group, we can reduce to the desired case by passing to its adjoint group and applying the property (iii) in the Main Theorem. Note that we do not need to set up injections as what we did in \textsection{4.2} since we are using the property (iii) in a different direction.
\subsection{The trivial character of $PGL_{1,D}$ (Conclusion of Proof)} \ 

\bigskip
For this part, we consider the classification of anisotropic groups of adjoint type over local fields. From the general theory about reductive groups of adjoint type (for example, see \cite[Theorem 24.3]{Mil17} and paragraphs after that), we know that such a group can be written as a direct product of groups of the form $Res_{E/F}(G)$ where $G$ is a connected $E$-anisotropic group of adjoint type which is geometrically almost-simple and $E/F$ is a finite separable extension. Note that this is the only place where we need to consider also finite separable extensions $E$ over $F$. Geometrically almost-simple, connected, anisotropic groups of adjoint type are classified by Bruhat and Tits in \cite{BT87}. Originally Bruhat and Tits classified all connected, simply-connected, geometrically almost-simple anisotropic groups over local fields, and these groups are of the form $SL_{1,D}$. By passing to their adjoint groups, we get the corresponding result for groups of adjoint type and they are of the form $PGL_{1,D}$. Now we can finish the last part by applying the compatibility with direct product (property (iv)) and the compatibility with Weil restriction of scalars (property (v)). For the uniqueness statement involving property (vii), one only needs to consider the exact sequence $1\rightarrow \mathbb{G}_m\rightarrow GL_1(D)\rightarrow PGL_1(D)\rightarrow 1$ and applies property (iii).

\begin{rmk} \label{weaken}
Note that all the properties mentioned in the Main Theorem have been used in the proof. However, from our proof, actually some assumptions in the Main Theorem can be weakened slightly. For example, we only need that properties (ii)-(v) hold for groups which are anisotropic mod center and we only require that property (vi) holds with respect to minimal Levi subgroups. For simplicity, we keep the assumptions in the Main theorem.
\end{rmk}


 \section{The compatibility with parahoric Satake parameters: General case}\label{Compatibility_general}
 
We start by reviewing quickly the construction of $s(\pi)$ for non-quasi-split groups and we use the same notations as in \textsection 2. More details can be found in \cite[\textsection 8,9]{Hai15}. For a non-quasi-split group $G$, we use the normalized transfer homomorphism to pass to its quasi-split inner form. Let $G^*$ be the unique quasi-split inner form of $G$ defined over $F$. There exists a pair $(G,\Psi)$ where $\Psi$ is a $\Gamma_F$-stable $G_{ad}^*(\bar{F})$-orbit of $\bar{F}$-isomorphisms $\psi : G\rightarrow G^*$. The isomorphism classes of such pairs are classified by $H^1(F,G_{ad}^*)
$. To distinguish Galois actions on $G$ and $G^*$, we use $\Gamma_F^*$, $I_F^*$ and $\Phi_F^*$ when we consider the Galois action on $G^*$. In order to formulate the transfer homomorphisms, we need to make a choice of some data, and in the end it turns out that our transfer homomorphisms do not depend on the choice we made.

Pick a maximal $F$-split torus $A$ of $G$, and let $M$ be the centralizer of $A$ in $G$ and let $P$ be an $F$-parabolic subgroup containing $M$ as a Levi factor. We also pick similar data $A^*$, $T^*$, and $B^*$ for $G^*$. Since $G^*$ is quasi-split, $T^*$ is a maximal $F$-torus and $B^*$ is an $F$-Borel subgroup. It turns out that there exits a unique standard (in the sense of containing $B^*$) $F$-parabolic subgroup $P^*\subseteq G^*$ such that $P^*$ is $G^*(\bar{F})$-conjugate to $\psi(P)$ for all $\psi\in\Psi$. We denote by $M^*$ the standard Levi factor of $P^*$. We can consider the nonempty $\Gamma_F$-stable subset $\Psi_M\subseteq \Psi$ consisting of isomorphisms which send $P$ to $P^*$ and also send $M$ to $M^*$. Now $\Psi_M$ is a nonempty $\Gamma_F$-stable $M^*_{ad}(\bar{F})$-orbit of $\bar{F}$-isomorphism $M\rightarrow M^*$. Pick any $\psi_0\in \Psi_M$ which is $F^{un}$-rational. Since $A$ is central in $M$, $\psi_0|_A$ is defined over $F$, and thus $\psi_0(A) \subseteq A^*$.

We can choose a maximal torus $\widehat{T^*}$ and a Borel subgroup $\widehat{B^*}$ of $\widehat{G^*}$ such that they are both $\Gamma^*_F$-stable. Now $\psi_0$ induces a $\Gamma_F$-equivariant map between the centers of the dual groups
$$
\widehat{\psi_0}:Z(\widehat{M})\cong Z(\widehat{M^*}) \hookrightarrow \widehat{T^*}.
$$
Furthermore, it also induces the following map (by abuse of notation, we still denote the map by $\widehat{\psi_0}$)
$$
\widehat{\psi_0}:(Z(\widehat{M}^{I_F}))_{\Phi_F}/W(G,A)\rightarrow (\widehat{T^*}^{I_F^*})_{\Phi_F^*}/W(G^*,A^*).
$$
Now given any weakly unramified character $\chi$ of $M(F)$ (recall that we can view it as an element in $(Z(\widehat{M}^{I_F}))_{\Phi_F}$), we send it to the image of $\delta^{-\frac{1}{2}}_{B^*}\widehat{\psi_0}(\delta^{\frac{1}{2}}_P\chi)$ under the map
$$
(\widehat{T^*}^{I_F^*})_{\Phi_F^*}/W(G^*,A^*)\rightarrow\, ^L\!G^*/\widehat{G^*}=\,^L\!G/\widehat{G}.
$$
Here $^L\!G^*/\widehat{G^*}$ is identified with $^L\!G/\widehat{G}$ via $\Psi$. By $\delta_{B^*}$ and $\delta_P$ we mean the usual modular characters used to define the normalized parabolic induction, and they are viewed as weakly unramified characters in our setting.

A short restatement of the above construction is that we make a suitable choice of an $\bar{F}$-isomorphism to the unique inner split form and use it to transfer the weakly unramified characters with some twist involved. Of course one needs to check that the construction is independent of the choice made above and $s(\pi) \in [\widehat{G}^{I_F}\rtimes \Phi_F]_{\textup{ss}}/\widehat{G}^{I_F}$. For those results, we refer the reader to \cite[\textsection 8,9]{Hai15} and note that the construction of $s(\pi)$ does not rely on the choice of the geometric Frobneius element $\Phi_F$.

The proof of Corollary B is basically the same as the proof of the Main Theorem if we know that the corresponding properties also hold for $s(\pi)$.

\begin{lem} \label{sat_lem}
The construction of parahoric Satake parameters $\pi \mapsto s(\pi)$ satisfies all the properties mentioned in the Main Theorem.
\end{lem}

\begin{proof}
Let us keep the notations as in the above construction. For property (i), this is Lemma \ref{tori_lemma}. For property (vi) (for the proof of Corollary B, it suffices to check this with respect to minimal Levi subgroups), we know that when $G$ is quasi-split, this is exactly the definition of $s(\pi)$. In general, this follows from the formula $\delta^{-\frac{1}{2}}_{B^*_{M^*}}=\delta^{-\frac{1}{2}}_{B^*}\widehat{\psi_0}(\delta^{\frac{1}{2}}_P)$ used in \cite[Lemma 11.1]{Hai15}. For property (vii), this is \cite[13.2]{Hai15}. For property (iii), one can check that, in this case, a minimal Levi subgroup of $G'$ is getting mapped to one in $G$. Then the rest follows from the definition of parahoric Satake parameters and the fact that the supercuspidal supports of parahoric-spherical representations are of the desired forms as claimed in the first step of the proof of the Main Theorem. For all the remaining properties (again for the proof of Corollary B, it suffices to check them for groups which are anisotropic mod center), the strategies are the same. Basically, one picks the $\bar{F}$-isomorphism compatibly and checks how the twist behaves. For example, for property (iv), when we deal with $M\times M'$, we choose $\bar{F}$-isomorphisms respectively for $M$ and $M'$, then we get an $\bar{F}$-isomorphism for $M\times M'$. Since the twist term can also be written in the form of direct products, we are done with property (iv). Among all the properties, the compatibility with Weil restriction of scalars is the hardest one and we give a proof of it. Other remaining properties are easier and can be proved in a similar manner. Without loss of generality, we assume that $E/F$ is a finite Galois extension of degree $m$ and $M$ is a anisotropic-mod-center group over $E$. One can check immediately that $Res_{E/F}(M)$ is also anisotropic mod center and that $Res_{E/F}(M^*)$ is the unique quasi-split inner form of $Res_{E/F}(M)$. Let $\pi_E$ be a weakly unramified character of $M(E)$ and we denote by $\pi_F$ the corresponding character (which is also weakly unramified \cite[Lemma 4.4]{HR20}) of $Res_{E/F}(M)(F)=M(E)$. We want to show that $s(\pi_E)$ uniquely determines $s(\pi_F)$ in a way which is compatible with Weil restriction of scalars. It suffices to consider the following two cases. We can pick $\Phi_E$ and $\Phi_F$ compatibly in the following two cases since the construction does not reply on the choice of the geometric Frobenius element.
\smallskip

\textit{Case 1:} $E/F$ is totally ramified. In this case $\Phi_E=\Phi_F$. Since $I_F$ acts on $(\widehat{Z(Res_{E/F}(M)}))^{I_E} \cong \prod_{g\in Gal(E/F)} (Z(\widehat{M}))^{I_E} $ as transitive permutations on copies of $Z(\widehat{M})$, we have the following commutative diagram.

\[\begin{tikzcd}
((Z(\widehat{M}))^{I_E})_{\Phi_E} \arrow{r}{\widehat{\psi_0}} \arrow{d}{\Delta} & ((Z(\widehat{M^*})^{I_E}))_{\Phi_E} \arrow{d}{\Delta} \arrow{r} & ((\widehat{T^*})^{I_E})_{\Phi_E}\arrow{d}{\Delta}\\
((\widehat{Z(Res_{E/F}(M)}))^{I_F})_{\Phi_F}  \arrow{r}{\widehat{\psi_0^m}} &((\widehat{Z(Res_{E/F}(M^*)}))^{I_F})_{\Phi_F} \arrow{r}& ((\widehat{Z(Res_{E/F}(T^*)}))^{I_F})_{\Phi_F}
\end{tikzcd}
\]
Here $\Delta$'s are the maps induced from diagonal embeddings $$Z(\widehat{M}) \rightarrow  \widehat{Z(Res_{E/F}(M)}) \cong \prod_{g\in Gal(E/F)} Z(\widehat{M}) $$ (we use similar definitions for $\widehat{M^*}$ and $\widehat{T^*}$). Note that $\Delta$'s are isomorphisms. When we view $\pi_E$ as an element in $((Z(\widehat{M}))^{I_E})_{\Phi_E}$ via the Kottwitz homomorphism and view $\pi_F$ as an element in $((\widehat{Z(Res_{E/F}(M)}))^{I_F})_{\Phi_F}$, we know that $\Delta(\pi_E)=\pi_F$ in $((\widehat{Z(Res_{E/F}(M)}))^{I_F})_{\Phi_F}$ (this is the compatibility of the Kottwitz homomorphism with the Weil restriction of scalars and it follows from the proof of \cite[Lemma 4.4]{HR20}). Note that since the twist for $\pi_F$ can be written as a direct product of $m$ copies of the twist for $\pi_E$, we know that $s(\pi_F)=\Delta(s(\pi_E))$ where $\Delta$ is the map induced from the diagonal embedding $\widehat{M} \rightarrow \widehat{Res_{E/F}(M)} \cong \prod_{g\in Gal(E/F)} \widehat{M}$ in a natural way. Thus in this case, $s(\pi_E)$ uniquely determines $s(\pi_F)$ in a compatible way.

\medskip

\textit{Case 2:} $E/F$ is unramified. In this case $\Phi_E=\Phi_F^m$ and $I_E=I_F$. We have a similar commutative diagram.

\[\begin{tikzcd}
((Z(\widehat{M}))^{I_E})_{\Phi_E} \arrow{r}{\widehat{\psi_0}} \arrow{d}{\alpha} & ((Z(\widehat{M^*})^{I_E}))_{\Phi_E} \arrow{d}{\alpha} \arrow{r} & ((\widehat{T^*})^{I_E})_{\Phi_E}\arrow{d}{\alpha}\\
((\widehat{Z(Res_{E/F}(M)}))^{I_F})_{\Phi_F}  \arrow{r}{\widehat{\psi_0^m}} &((\widehat{Z(Res_{E/F}(M^*)}))^{I_F})_{\Phi_F} \arrow{r}& ((\widehat{Z(Res_{E/F}(T^*)}))^{I_F})_{\Phi_F}
\end{tikzcd}
\]
Here $\alpha$'s are the maps induced from the embedding $$Z(\widehat{M}) \rightarrow  \widehat{Z(Res_{E/F}(M)}) \cong \prod_{g\in Gal(E/F)} Z(\widehat{M}) $$ (again we use similar definitions for $\widehat{M^*}$ and $\widehat{T^*}$) into the component corresponding to the identity element in $Gal(E/F)$. Note that $\alpha$'s are isomorphisms. Again via the Kottwitz homomorphism, we have $s(\pi_F)=\alpha(s(\pi_E))$. With this identification, we know that $s(\pi_E)$ uniquely determines $s(\pi_F)$ in a compatible way. This finishes the proof of the compatiblity with Weil restriciton of scalars.

\end{proof}

Note that Corollary B does not follow directly from the Main Theorem for $\pi \mapsto \varphi^{ss}_{\pi}(\Phi_F)$, because a priori we do not know whether $\varphi^{FS}_{\pi}(\Phi_F) \in [\widehat{G}^{I_F}\rtimes \Phi_F]_{\textup{ss}}/\widehat{G}^{I_F}$ or not. For that, we need a result on $\varphi^{FS}_{\pi}(I_F)$. It is expected, when $\pi$ has non-trivial parahoric fixed vectors, that there is an element $\varphi$ in the $\widehat{G}$-conjugacy class of $\varphi^{FS}_{\pi}$ such that $\varphi(x)=1\rtimes x$ for all $x\in I_F$. For $GL_n$, this is known and in general it is an expected property of local Langlands correspondence. For general $G$, it is the following theorem.

\begin{thm} \label{inertia_thm}
Assume that $F$ is a $p$-adic field. Let $G$ be a connected reductive group defined over $F$ and $\pi$ be a smooth irreducible representation of $G(F)$ with non-trivial parahoric fixed vectors. Then there is an element $\varphi$ in the $\widehat{G}$-conjugacy class of $\varphi^{FS}_{\pi}$ such that $\varphi(x)=1\rtimes x$ for all $x\in I_F$.
\end{thm}

\begin{proof}
The proof is the same as the proof of the Main Theorem once we set up a similar fact as the injectivity in Lemma \ref{inj_lem}. Let $1\rightarrow Z \rightarrow G'\rightarrow G \rightarrow 1$ be a $z$-extension. We need to show that the sequence $1 \rightarrow H^1(I_F,\widehat{G}) \rightarrow H^1(I_F,\widehat{G'})$ is exact. Note that the exactness is weaker than the injectivity, but that is sufficient for our proof. From the exact sequence
$$1 \rightarrow  \widehat{G} \rightarrow \widehat{G'} \rightarrow \widehat{Z} \rightarrow 1,$$
it is equivalent to show that the map $\widehat{G'}^{I_F} \rightarrow \widehat{Z}^{I_F}$ is a surjection. It suffices to show that the map $Z(\widehat{G'})^{I_F} \rightarrow \widehat{Z}^{I_F}$ is a surjection. Recall that $Z$ is an induced torus, so $\widehat{Z}$ is a direct product of copies of $\mathbb{C}^{\times}$ and $I_F$ acts on $\widehat{Z}$ via permutations on those copies of $\mathbb{C}^{\times}$, and this implies that $\widehat{Z}^{I_F}$ is still a direct product of copies of $\mathbb{C}^{\times}$. Suppose that $I_F$ acts through a finite quotient $Q$ on the exact sequence $1 \rightarrow  Z(\widehat{G}) \rightarrow Z(\widehat{G'}) \rightarrow \widehat{Z} \rightarrow 1$ and $Q$ is of order $m$. 
For any $x\in \widehat{Z}^{I_F}$, we can pick $y\in Z(\widehat{G'})$ a preimage of $x$. Let $\breve{y}= \Pi_{g\in Q} \ g(y)$, then $\breve{y}\in Z(\widehat{G'})^{I_F}$ and its image in $\widehat{Z}$ is $x^m$. Since $\mathbb{C}^{\times}$ is divisible, this implies that the map $Z(\widehat{G'})^{I_F} \rightarrow \widehat{Z}^{I_F}$ is a surjection.

Note that the proof relies on the corresponding result (the case we finally reduce to) for the trivial representation of $D^\times$ which has the same Langlands parameter as the Steinberg representation of its quasi-split inner form $GL_n(F)$ (for this fact, see the proof of \cite[Lemma 13.4]{Hai15} which relies further on \cite[Definition 4.1 and Theorem 4.6]{Coh18}, \cite[\textsection 7.2]{Bad01} and \cite{Bad07}). Thus our proof ultimately follows from the corresponding result for $GL_n$, and there it comes from local class field theory.
\end{proof}

Let us go back to the proof of Corollary B. By Theorem \ref{inertia_thm}, we know that the image of $\varphi^{FS}_{\pi} \in H^1(W_F,\widehat{G})$ under the restriction map to $H^1(I_F,\widehat{G})$ is trivial. By the inflation-restriction exact sequence, we know that $\varphi^{FS}_{\pi}(\Phi_F) \in [\widehat{G}^{I_F}\rtimes \Phi_F]/\widehat{G}^{I_F}= H^1(\langle \Phi_F\rangle,\widehat{G}^{I_F})$. Combining this fact and Lemma \ref{sat_lem}, we can write down the proof of Corollary B verbatim from the proof of the Main Theorem except that we also need to set up a similar version of Lemma \ref{inj_lem}. We want to prove that, for any $z$-extension $1\rightarrow Z \rightarrow G'\rightarrow G \rightarrow 1$, the induced map $H^1(\langle \Phi_F\rangle,\widehat{G}^{I_F}) \rightarrow H^1(\langle \Phi_F\rangle,\widehat{G'}^{I_F})$ is an injection. But this follows directly from Lemma \ref{inj_lem} since $H^1(\langle \Phi_F\rangle,\widehat{G}^{I_F})$ (respectively $H^1(\langle \Phi_F\rangle,\widehat{G'}^{I_F})$) is a subset of $H^1(W_F,\widehat{G})$ (respectively $H^1(W_F,\widehat{G'})$). We are done with Corollary B. As a result, now we know that $\varphi^{FS}_{\pi}(\Phi_F) \in [\widehat{G}^{I_F}\rtimes \Phi_F]_{\textup{ss}}/\widehat{G}^{I_F}$.

\begin{rmk}
Corollary B and Theorem \ref{inertia_thm} justify the definition of semisimple local $L$-functions for parahoric-spherical representations in a recent paper \cite[Definition 4.4 and Remark 4.5]{OST22}. That definition also agrees with the definition given in \cite[Definition 9.8]{Hai05}.
\end{rmk}

Finally, let us emphasize again that the assumption that $F$ is a $p$-adic field is only needed for applying the results in \cite{HKW22}.

\end{document}